\documentclass[12pt]{amsart}
\usepackage{amssymb,verbatim,enumerate,ifthen}
\usepackage[mathscr]{eucal}
\usepackage[utf8]{inputenc}
\usepackage[T1]{fontenc}
\usepackage{esint}
\usepackage{marginnote}
\newtheorem{thm}{Theorem}[section]
\newtheorem{cor}[thm]{Corollary}
\newtheorem{prop}[thm]{Proposition}
\newtheorem{lem}[thm]{Lemma}

\newtheorem{conj}{Conjecture}

\theoremstyle{definition}
\newtheorem{defin}[thm]{Definition}
\newtheorem{rem}[thm]{Remark}
\newtheorem{exa}[thm]{Example}

\numberwithin{equation}{section}
\def\eq#1{{\rm(\ref{#1})}}
\def\Eq#1#2{\ifthenelse{\equal{#1}{*}}
  {\begin{equation*}\begin{aligned}[]#2\end{aligned}\end{equation*}}
  {\begin{equation}\begin{aligned}[]\label{#1}#2\end{aligned}\end{equation}}}

\DeclareUnicodeCharacter{1EF3}{\`y}

\def\A{\mathscr{A}}

\def\G{\mathscr{G}}
\def\M{\mathscr{M}}
\def\P{\mathscr{P}}
\newcommand{\operator}[1]{\mathop{\vphantom{\sum}\mathchoice
{\vcenter{\hbox{\LARGE $#1$}}}
{\vcenter{\hbox{\Large $#1$}}}{#1}{#1}}\displaylimits}

\def\Mm{\operator{\mathscr{M}}}
\def\Ar{\operator{\mathscr{A}}}       
\newcommand\R{\mathbb{R}}
\newcommand\N{\mathbb{N}}
\newcommand\Z{\mathbb{Z}}
\newcommand\Q{\mathbb{Q}}

\newcommand{\norm}[1]{\left\| #1 \right\| }
\newcommand{\abs}[1]{\left| #1 \right| }
\newcommand{\ceil}[1]{\left\lceil #1 \right\rceil}

\newcommand{\vone}{\textbf1}
\newcommand{\pto}{\xrightarrow{p}}

\DeclareMathOperator{\lcm}{lcm}

\newcommand{\Hc}[2][SKIPPED]{
\ifthenelse{\equal{#2}{}}
  {
    \ifthenelse{\equal{#1}{SKIPPED}}
      {\mathscr{H}}
      {\mathscr{H}(#1)}
  }
  {
    \ifthenelse{\equal{#1}{SKIPPED}}
      {\mathscr{H}_{#2}}
      {\mathscr{H}_{#2}(#1)}
  }
}
\newcommand{\Est}[2][SKIPPED]{
\ifthenelse{\equal{#1}{SKIPPED}}
  {
    \ifthenelse{\equal{#2}{}}
      {\mathscr{C}}
      {\mathscr{C}(#2)}
  }
  {
    \ifthenelse{\equal{#2}{}}
      {\mathscr{C}_{#1}}
      {\mathscr{C}_{#1}(#2)}
  }
}
\author[P. Pasteczka]{Pawe\l{} Pasteczka}
\address{Institute of Mathematics \\ Pedagogical University of Krak\'ow \\ Podchor\k{a}\.zych str. 2, 30-084 Krak\'ow, Poland}
\email{pawel.pasteczka@up.krakow.pl}

    

\title{On properties of weighted Hardy constant for means}
\begin{document}
\begin{abstract}
For a given weighted mean $\mathscr{M}$ defined on a subinterval of $\mathbb{R}_+$ and a sequence of weights  $\lambda=(\lambda_n)_{n=1}^\infty$ we define a Hardy constant $\mathscr H(\lambda)$ as the smallest extended real number such that
$$ \sum_{n=1}^\infty \lambda_n \mathscr{M}\big((x_1,\dots,x_n),(\lambda_1,\dots,\lambda_n)\big) \le \mathscr H(\lambda) \cdot \sum_{n=1}^\infty \lambda_n x_n \text{ for all }x \in \ell^1(\lambda).$$

The aim of this note is to present a comprehensive study of the mapping $\mathscr H$. For example we prove that it is lower semicontinuous in the pointwise topology. 

Moreover we show that whenever $\mathscr{M}$ is a monotone and Jensen-concave mean which is continuous in its weights then $\mathscr H$ is monotone with respect to the partitioning of the vector. Finally we deliver some sufficient conditions for $\lambda$ to validate the equality $\mathscr H(\lambda)=\sup \mathscr H$ for every symmetric and monotone mean. 
\end{abstract}
\maketitle


\section{Introduction}
History of Hardy inequality began in 1920s and a series of papers by Hardy \cite{Har20}, Landau \cite{Lan21}, Knopp \cite{Kno28}, and Carleman \cite{Car32}. Their results can be summarized as the 
inequality involving the $p$-th power mean $\P_p$. More precisely they proved that
\Eq{E:C(p)0}{
\sum_{n=1}^\infty \P_p(x_1,\dots,x_n) < C(p) \cdot \sum_{n=1}^\infty x_n 
}
for all $p<1$ and $x \in \ell^1(\R_+)$, where 
\Eq{E:C(p)}{
C(p):=
\begin{cases} 
(1-p)^{-1/p}&p \in (-\infty,0) \cup (0,1), \\ 
e & p=0.
\end{cases} 
}
Moreover it is known that the above constants are sharp. It can be extended by putting $C(-\infty):=1$ and $C(p):=+\infty$ for all $p \in [1,+\infty]$ (with a natural extension of power means $\P_{-\infty}=\min$ and $\P_{+\infty}=\max$). For more details we refer the reader to surveys by Pe\v{c}ari\'c--Stolarsky \cite{PecSto01}, Duncan--McGregor \cite{DunMcg03}, and a book of Kufner--Maligranda--Persson \cite{KufMalPer07}.

This classical result was extended in several directions. First, P\'ales and Persson \cite{PalPer04} introduced a notion of Hardy mean. More precisely $\M\colon \bigcup_{n=1}^\infty I^n \to I$ (here $I$ stands for an interval with $\inf I=0$) is a \emph{Hardy mean} if there exists a constant $C \in (0,+\infty)$ such that
\Eq{*}{
\sum_{n=1}^\infty \M(x_1,\dots,x_n) \le C \cdot \sum_{n=1}^\infty x_n \qquad \text{ for all }x \in \ell^1(I).
}

In the next step, following the notion from \cite{PalPas16}, the smallest extended real number $C$ satisfying this inequality is called a \emph{Hardy constant} of $\M$ and denoted here simply by $H$. In this setup a mean is a Hardy mean if and only if its Hardy constant is finite.

In fact the most important result from \cite{PalPas16} is that whenever $\M$ is a monotone, symmetric, Jensen concave, homogeneous, and repetition invariant mean on $\R_+$ then its Hardy constant is given by a limit
\Eq{*}{
H =\lim_{n \to \infty} n \cdot \M\big(1,\tfrac12,\dots,\tfrac1n\big).
}
In particular this sequence is always convergent (possibly to $+\infty$) and $\M$ is a Hardy mean if and only if this limit is finite.  This result generalized the inequality from 1920s. 

The next step was to deliver a weighted counterpart of Hardy inequality. Such generalization was first study by Copson \cite{Cop27} and Elliott \cite{Ell26} who proved the inequality
\Eq{*}{
\sum_{n=1}^\infty \P_p \big( (x_1,\dots,x_n),(\lambda_1,\dots,\lambda_n)\big) < C(p) \sum_{n=1}^\infty \lambda_nx_n
}
for every $p\in(0,1)$, all-positive-entries sequence $\lambda$, and $x \in \ell^1(\lambda)$ (here $\P_p$ stands for the weighted $p$-th power mean). This result is generalized in a series of papers by P\'ales and Pasteczka \cite{PalPas18b,PalPas18a,PalPas19b,PalPas19a,PalPas20}.

All precise definitions concerning weighted means are given in the next section. Let us now give some insight into these results. 

One of new concepts which appeared in \cite{PalPas19b} was to introduced a weighted Hardy constant. For a weighted mean $\M$ (see the next section for the definition) and infinite sequence of weights $\lambda$ we define $\Hc{}(\lambda)$ as the smallest extended real number such that 
\Eq{*}{
\sum_{n=1}^\infty \lambda_n \M\big((x_1,\dots,x_n),(\lambda_1,\dots,\lambda_n)\big) \le \Hc{}(\lambda) \cdot \sum_{n=1}^\infty \lambda_nx_n \qquad \text{ for all }x \in \ell^1(\lambda).
}

Note that for $\lambda=(1,1,\dots)=:\vone$ we go back to the nonweighted setting, thus we have $\Hc{}(\vone)=H$ (we recall some of these definitions more precisely in section~\ref{sec:DefsHardy}).

Remarkably, it turned out that whenever $\M$ is monotone and symmetric then the maximal weighted Hardy constant is a nonweighted one (which refers to a constant sequence $\lambda$) -- cf. \cite[Theorem 2.8]{PalPas19b} which is quoted in Theorem~\ref{thm:mu1} below. This obviously extends the Copson--Elliott result.

Second important result states that whenever $\M$ is symmetric, monotone, and Jensen-concave weighted mean (either $\R$-weighted which is continuous in its weights or $\Q$-weighted), and $(\lambda_n)_{n=1}^\infty$ is a sequence of weights such that $\sum_{n=1}^\infty \lambda_n=+\infty$ and $(\frac{\lambda_n}{\lambda_1+\dots+\lambda_n})_{n=1}^\infty$ is nonincreasing then
\Eq{E:Hlambdalim}{
\Hc{}(\lambda)=\sup_{y>0} \liminf_{n \to \infty} \tfrac {\lambda_1+\lambda_2+\cdots+\lambda_n}y \cdot \M \Big(\big(\tfrac{y}{\lambda_1},\tfrac{y}{\lambda_1+\lambda_2},\tfrac{y}{\lambda_1+\lambda_2+\lambda_3},\dots\big),(\lambda_1,\lambda_2,\lambda_3,\dots)\Big).
}
The key tool of the proof was so-called (nonweighted) Kedlaya  inequality \cite{Ked94} and its weighted counterpart \cite{Ked99}, which was extended in both of these cases -- cf. \cite{PalPas16} and \cite{PalPas18b}, respectively.

Having this, our consideration split into two parts. First issue was to characterize Jensen-concavity for vary families of means (symmetry and monotonicity are simpler in general) -- such results are contained in \cite{PalPas19a}. Second problem was to calculate a weighted Hardy constant for some particular families (and weights) -- it was done in \cite{PalPas20}.

\medskip

The present paper is a continuation of this research. 
We study the properties of the mapping $\Hc{}$ for a given (fixed) mean. This problem arised from the paper \cite{PalPas19b}, where it was proved that, under some additional assumptions, the maximal value is obtained for the contstant vector. There were also few other results (for particular families of means and under some additional condition on~$\lambda$) which arised from studying the equality \eq{E:Hlambdalim} (see \cite{PalPas20}). All of them can be encompassed in the following form: \emph{Under certain conditions the value $\Hc{}(\lambda)$ depends on $\lambda$ only implicitly by a limit of the ratio sequence $\big(\frac{\lambda_n}{\lambda_1+\dots+\lambda_n}\big)_{n=1}^\infty$}. It~also corresponds to our Theorem~\ref{thm:main}.

\section{Weighted means}

In this section we recall several preliminary results concerning weighted means. This definition first appeared in \cite{PalPas18b} in the context of so-called Kedlaya inequality \cite{Ked94,Ked99}. It is separated from any particular family of means, which was a new idea.

\begin{defin}[\cite{PalPas18b}, Weighted mean]\label{def:WQA}
Let $I \subset \R$ be an arbitrary interval, $R \subset \R$ be a ring and, for 
$n\in\N$, define the set of $n$-dimensional weight vectors $W_n(R)$ by
\Eq{*}{
  W_n(R):=\{(\lambda_1,\dots,\lambda_n)\in R^n\mid\lambda_1,\dots,\lambda_n\geq0,\,\lambda_1+\dots+\lambda_n>0\}.
}
\emph{A weighted mean on $I$ over $R$} or, in other words, \emph{an $R$-weighted mean on $I$} is a function 
\Eq{*}{
\M \colon \bigcup_{n=1}^{\infty} I^n \times W_n(R) \to I
}
satisfying the following conditions:

\begin{enumerate}[(i)]
 \item \emph{Nullhomogeneity in the weights}: For all $n \in \N$, for all $(x,\lambda) \in I^n \times W_n(R)$, 
and $t \in R_+$,
 \Eq{*}{
   \M(x,\lambda)=\M(x,t \cdot \lambda),
 }
\item \emph{Reduction principle}: For all $n \in \N$ and for all $x \in I^n$, $\lambda,\mu \in W_n(R)$, 
\Eq{*}{
\M(x,\lambda+\mu)=\M(x\odot x,\lambda\odot\mu),
}
where $\odot$ is a \emph{shuffle operator}, that is $
(p_1,\dots,p_n)\odot (q_1,\dots,q_n):=(p_1,q_1,\dots,p_n,q_n)$.

\item \emph{Mean value property}: For all $n \in \N$, for all $(x,\lambda) \in I^n \times W_n(R)$
\Eq{*}{
\min(x_1,\dots,x_n) \le \M(x,\lambda)\le \max(x_1,\dots,x_n),
}
\item \emph{Elimination principle}: 
entries with a zero weight can be omitted. 
\end{enumerate}
\end{defin}

From now on $I$ is an arbitrary interval, $R$ stands for an arbitrary subring of $\R$.

Following \cite{PalPas18b}, a weighted mean $\M$ is said to be \emph{symmetric}, if for all $n \in \N$, $x \in I^n$, $\lambda \in W_n(R)$, and a permutation $\sigma \in S_n$ we have
$\M(x,\lambda) =\M(x\circ\sigma,\lambda\circ\sigma)$. 
$\M$ is called \emph{monotone} if it is nondecreasing in each of its entry. Similarly $\M$ is \emph{concave} if for every $n \in \N$ and $\lambda \in W_n(R)$ the mapping $I^n \ni x \mapsto \M(x,\lambda) \in I$ is concave (or equivalently, by~\cite{BerDoe15}, Jensen concave).

In fact in can be proved that every $R$-weighted mean admits a unique extension to $R^*$-weighted mean ($R^*$ stands for the quotient field, i.e. the smallest field generated by $R$). Moreover this extension preserve few important properties. This statement binds few results \cite[Theorems 2.2--2.5]{PalPas18b}.

\begin{prop}
\label{prop:fieldex}
Let $I$ be an interval, $R \subset \R$ be a ring, $\M$ be a weighted mean defined on $I$ over $R$.
Then there exists a unique mean $\widetilde\M$ defined on $I$ over $R^*$ (which denotes the quotient field of $R$) such that 
\Eq{*}{
\widetilde \M \vert_{\bigcup_{n=1}^{\infty} I^n \times W_n(R)} =\M.
}
Moreover if $\M$ is symmetric, monotone or Jensen concave then so is $\widetilde\M$, respectively.
\end{prop}

This proposition is of essential importance as it allows to extend nonweighted means to weighted ones.
Indeed, there exists a natural correspondence between repetition invariant means and $\Z$-weighted means (see \cite[Theorem~2.3]{PalPas18b} for details). Then by Proposition~\ref{prop:fieldex} it can be uniquely extended to $\Q$-weighted mean and, whenever there exists a continuous extension, to $\R$-weighted mean. What is more, for a given mean such extension is uniquely determined and in most cases it coincide with  already known generalizations -- for example for quasideviation means \cite{Pal88a} and all its subclasses: quasiarithmetic means \cite{HarLitPol34}, Gini means \cite{Gin38}, Bajraktarevi\'c means \cite{Baj58,Baj63}, deviation (Dar\'oczy) \cite{Dar72b} means and so on.

Based on these facts and nullhomogeneity in the weights hereafter we assume $1 \in R$. It can be easily checked that the arithmetic mean (from now on denoted by $\A$) is an $\R$-weighted mean on $\R$.

\subsection{$R$-simple functions. Sum-type and integral-type notation}
For the sake of convenience, we will use the sum-type and integral-type abbreviation. First, if $\M$ is an $R$-weighted mean on $I$, $n\in\N$ and $(x,\lambda) \in I^n\times W_n(R)$, then we denote
\Eq{*}{
\Mm_{i=1}^n(x_i,\lambda_i):=\M\big((x_1,\dots,x_n),(\lambda_1,\dots,\lambda_n)\big).
}

To introduce the integral-type notion we need to define so-called $R$-intervals. We say that $D\subseteq\R$ is an \emph{$R$-interval} if $D$ is of the form $[a,b)$ for some $a,b\in R$. For a given $R$-interval $D=[a,b)$, a function $f\colon D\to I$ is called \emph{$R$-simple} if there exist a partition of $D$ into a finite number of $R$-intervals $\{D_i\}_{i=1}^n$ such that:
\begin{enumerate}[(i)]
 \item $\sup D_i=\inf D_{i+1}$ for all $i\in\{1,\dots,n-1\}$ 
 \item $f$ is constant on each $D_i$.
 \end{enumerate} 
 Then, for an $R$-weighted mean $\M$ on $I$ and $R$-simple function $f$ like above, we define
\Eq{E:int_notion}{
\Mm_a^b f(x)dx:=\Mm_{i=1}^n \big(f|_{D_i},|D_i|\big)=\M\big((f|_{D_1},\dots,f|_{D_n}),(|D_1|,\dots,|D_n|)\big).
}
Let us just mention that we use reduction principle to define this function -- that is to guarantee that the value of a mean does not depend on a choice of $(D_i)$. 

In this setting $\M$ is symmetric if and only if for every pair of $R$-simple functions $f,g \colon D \to I$ which have the same
distribution the equality $\M f(x)dx=\M g(x)dx$ holds. Similarly $\M$ is monotone if and only if for every pair of $R$-simple functions $f,g \colon D \to I$ with $f\le g$ the inequality $\M f(x)dx \le \M g(x)dx$ is valid.

Let us introduce the notion of a weighted characteristic function. For $n \in \N_+ \cup\{\infty\}$, $x \in I^n$ and $\lambda \in [0,\infty)^n$ set $\Lambda_k:=\sum_{i=1}^k \lambda_i$ ($0\le k\le n$) (for $n\in\{0,+\infty\}$ we take a natural extension) and define $\chi_{x,\lambda} \colon [0, \Lambda_n)  \to I$ by 
\Eq{*}{
\chi_{x,\lambda}(t)=x_k \qquad \text{ for }t \in [\Lambda_{k-1},\Lambda_k) \qquad k \in \N \cap [1,n].
}
Observe that, in view of \eq{E:int_notion}, for every mean $R$-weighted mean $\M$ on $I$, $n \in \N$, and a pair $(x,\lambda)\in I^n \times W_n(R)$ we have following identities
\Eq{*}{
\M(x,\lambda)=\Mm_{i=1}^n(x_i,\lambda_i)
=\Mm_{i=1}^n\big(\chi_{x,\lambda}(\Lambda_{i-1}),\lambda_i\big)
=\Mm_{0}^{\Lambda_n} \chi_{x,\lambda}(t) dt.
}

\subsection{\label{sec:DefsHardy}Hardy inequality}

For the simplicity we will assume that weight zero is not allowed. Therefore let $W^0_N(R)=(R \cap (0,+\infty))^N$ and $W^0(R)=(R \cap (0,+\infty))^\infty$. Let us first recall the definition of weighted Hardy property which was already mentioned in the introduction.

\begin{defin}[\cite{PalPas19b}, Weighted Hardy property]
Let $I$ be an interval with $\inf I=0$, $R \subset \R$ be a ring. For an $R$-weighted mean $\M$ on $I$ and weights $\lambda \in W^0(R)$, let $C$ be the smallest extended real number such that for all sequences $(x_n)$ in $I$,
\Eq{E:WHI}{
\sum_{n=1}^{\infty} \lambda_n \cdot \Mm_{i=1}^n\big(x_i,\lambda_i\big) \le C \cdot \sum_{n=1}^{\infty} \lambda_nx_n.
}
We call $C$ the \emph{$\lambda$-weighted Hardy constant of $\M$} or the \emph{$\lambda$-Hardy constant of $\M$} and denote it by $\Hc\M(\lambda)$. Whenever this constant is finite, then $\M$ is called a \emph{$\lambda$-weighted Hardy mean} or simply a \emph{$\lambda$-Hardy mean}.

\end{defin}

In the other setup for $\lambda \in W^0(R)$ we consider a weighted $\ell^1$ space 
\Eq{*}{
\ell^1(\lambda,I):=\Big\{(x_1,x_2,\dots)\in I^\N\colon \norm{x}_{\ell_1(\lambda)}:=\sum_{n=1}^\infty \lambda_i\abs{x_i}<\infty\Big\}
}
words for a given $R$-weighted mean $\M$ on $I$ we define the weighted averaging operator $T_\M \colon \ell^1(\lambda,I) \to I^\N$ by
\Eq{*}{
T_{\M,\lambda}(x_1,x_2\dots)=\Big(\Mm_{i=1}^n\big(x_i,\lambda_i\big)\Big)_{n=1}^\infty.
}
Using this notation we have
\Eq{*}{
\Hc\M(\lambda)=\norm{T_{\M,\lambda}}_{\ell_1(\lambda)\to\ell_1(\lambda)}\qquad \text{ for all }\lambda \in W^0(R),} 
as, by the definition, $\Hc\M(\lambda)$ is the smallest extended real number $C$ such that
\Eq{*}{
\norm{T_{\M,\lambda}(x)}_{\ell_1(\lambda)} \le C \norm{x}_{\ell_1(\lambda)}\text{ for all }x \in \ell_1(\lambda).
}
Next result shows that under mild assumptions the maximal Hardy constant is the non-weighted one -- more precisely the one which is related to the vector $\vone:=(1,1,\dots)$.
\begin{thm}[\cite{PalPas19b}, Theorem 2.8]
\label{thm:mu1}
For every symmetric and monotone weighted mean $\M$ we have $\sup\Hc\M=\Hc\M(\vone)$.
\end{thm}

Let us now give some insight into \cite[section 5]{PalPas19b} which was completely devoted to the proof of this theorem. It was split into three, somewhat independent, statements which we recall below. It is quite easy to bind them to the final form.

\begin{lem}[\cite{PalPas19b}, Lemma 5.1]
\label{lem:decr}
Let $\M$ be a $R^*$-weighted, monotone mean on $I$ and $a \in R^* \cap (0,\infty)$. Then the mapping $R^*\cap(0,a] \ni u \mapsto \Mm_0^uf(t)\:dt \in I$ is nonincreasing for every nonincreasing $R^*$-simple function $f\colon [0,a) \to I$.
\end{lem}

\begin{lem}[\cite{PalPas19b}, Lemma~5.2]\label{lem:L5.2}
Let $\M$ be a monotone $R^*$-weighted mean on $I$. Then, for all $N \in \N$, for all nonincreasing sequences $x \in I^N$ and weights $\lambda\in W^0_N(R^*)$, the inequality 
\Eq{*}{
\sum_{n=1}^N \lambda_n \Mm_{i=1}^n \big(x_i,\lambda_i\big) \le \Hc\M(\vone) \sum_{n=1}^N \lambda_nx_n.
}
is valid.
\end{lem}

\begin{lem}[\cite{PalPas19b}, Lemma 5.3]
\label{lem:rearrangement}
Let $\M$ be a symmetric and monotone $R$-weighted mean on $I$. Then, for all $N\in\N$, for all vectors $x\in I^N$ and weights $\lambda \in W^0_N(R)$, there exist $M \in \N$, a nonincreasing sequence $y\in I^M$ and a weight sequence $\psi\in W^0_M(R)$ such that
\Eq{*}{
\sum_{n=1}^N \lambda_n \Mm_{i=1}^n \big(x_i,\lambda_i\big)&\le \sum_{m=1}^M \psi_m \Mm_{i=1}^m \big(y_i,\psi_i\big).
}
and
\Eq{*}{
\sum_{\{n\colon x_n=t\}} \lambda_n =\sum_{\{m\colon y_m=t\}} \psi_m
}
for all $t\in\R$. In particular
$\sum_{n=1}^N \lambda_nx_n=\sum_{m=1}^M \psi_my_m$.
\end{lem}

Next theorem shows that whenever the mean $\M$ admit some additional assumptions, we can prove a counterpart of this lemma with $\psi=\lambda$. However, the sequence $x$ and $y$ are no longer equidistributed so it cannot be considered as a generalization. We also need much more assumptions for the mean $\M$.

\begin{thm}
\label{thm:samesum}
 Let $\M$ be a monotone and Jensen concave
$\Q$-weighted mean on $I$ (resp. $\R$-weighted mean on $I$ which is continuous in its weights).

For every $\lambda \in W_N^0(\Q)$ (resp. $\lambda \in W_N^0(\R)$) and $x \in I^N$ there exists a nonincreasing sequence $y \in I^N$ such that $\sum_{n=1}^N \lambda_nx_n=\sum_{n=1}^N \lambda_ny_n$ and 
 \Eq{JCin}{
\Mm_{i=1}^n (x_i,\lambda_i) \le \Mm_{i=1}^n (y_i,\lambda_i) \text{ for all }n \in\{1,\dots,N\}.
}
 \end{thm}
\begin{proof}
First assume that $\lambda \in W_N^0(\Z)$.
Define $(s_k)_{k=1}^{\Lambda_N}$ by $s_k=x_n$ for $k \in \{\Lambda_{n-1}+1,\dots,\Lambda_n\}$. Let $(s^*_k)$ be a nondecreasing rearrangement of $(s_k)$ and define the sequence $(y_n)_{n=1}^N$ as 
\Eq{*}{
y_n:=\Ar_{k=\Lambda_{n-1}+1}^{\Lambda_n} s_k^*=\dfrac{s_{\Lambda_{n-1}+1}^*+\cdots+s^*_{\Lambda_n}}{\lambda_n},\quad n \in \{1,\dots,N\}.
} 

Obviously
\Eq{*}{
\sum_{n=1}^N \lambda_ny_n=\sum_{k=1}^{\Lambda_N} s_k^*=\sum_{k=1}^{\Lambda_N} s_k=\sum_{n=1}^N \lambda_nx_n.
}
Moreover, as both $\A$ and $(s^*_k)$ are monotone, then so is $(y_n)$. Furthermore as $\M$ is symmetric and monotone for all $n \in \{1,\dots,N\}$ we have
\Eq{*}{
\Mm_{i=1}^n (x_i,\lambda_i) =\Mm_{k=1}^{\Lambda_n} (s_k,1) \le \Mm_{k=1}^{\Lambda_n} (s_k^*,1). 
}

Now define a permutation $\pi \colon \{1,2,\dots,\Lambda_N\} \to \{1,2,\dots,\Lambda_N\}$ (in a cyclic notion) by 
\Eq{*}{
\pi:=(1,\dots,\Lambda_1)(\Lambda_1+1,\dots,\Lambda_2)\cdots(\Lambda_{n-1}+1,\dots,\Lambda_n)
}
Denote briefly the order of $\pi$ by $\pi_{\rm ord}=\lcm(\Lambda_1,\dots,\Lambda_N)$. 
Then, by Jensen-concavity and symmetry of $\M$, for all $n \in\{1,\dots,N\}$ we obtain 
\Eq{*}{
\Mm_{i=1}^n (x_i,\lambda_i) 
&\le \Mm_{k=1}^{\Lambda_n} (s_k^*,1)
=\frac 1{\pi_{\rm ord}} \sum_{j=1}^{\pi_{\rm ord}} \Mm_{k=1}^{\Lambda_n} \big(s_{\pi^j(k)}^*,1\big)
\le  \Mm_{k=1}^{\Lambda_n} \Big(\frac 1{\pi_{\rm ord}} \sum_{j=1}^{\pi_{\rm ord}} s_{\pi^j(k)}^*,1\Big)\\
&=  \Mm_{k=1}^{\Lambda_n} \Big(\Ar_{j=1}^{\pi_{\rm ord}} s_{\pi^j(k)}^*,1\Big)
= \Mm_{i=1}^n\Big(\Ar_{k=\Lambda_{i-1}+1}^{\Lambda_i} s_k^*,\lambda_i\Big)
= \Mm_{i=1}^n (y_i,\lambda_i),
}
which is \eq{JCin}. 

For $\lambda \in W_N^0(\Q)$ there exists a natural number $K \in \N$ such that $K\lambda \in W_N^0(\Z)$. Then \eq{JCin} is true for a triple $(x,y,K\lambda)$ which, by nullhomogeneity in weights implies that it is remains valid for a triple $(x,y,\lambda)$, too.

Finally, if $\M$ is $\R$-weighted mean which is continuous in its weights then, applying above consideration, we obtain that \eq{JCin} is valid for all $\lambda \in W_N^0(\Q)$. However in this case both sides of \eq{JCin} are continuous in $\lambda$, thus inequality \eq{JCin} can be extended to whole $W_N^0(\R)$.
\end{proof}

Let us now recall two technical results concerning divergence of sequences.

\begin{lem}[\cite{PalPas20}, Lemma 4.1] 
\label{lem:equicon}
The sequence $(\Lambda_n)$ and the series $\sum\lambda_n/\Lambda_n$ are equi-convergent (either both of them are convergent or both of them are divergent).
\end{lem}

\begin{lem}[\cite{PalPas20}, Lemma 4.2]
\label{lem:maxk1n}
If $\lambda_n/\Lambda_n \to 0$ and $\Lambda_n\to\infty$, then 
\Eq{*}{
\lim_{n \to \infty} \frac{\max(\lambda_1,\dots,\lambda_n)}{\Lambda_n} =0.
}
\end{lem}

Next lemma a generalization of \cite[Proposition~3.1]{PalPas16} where it was stated in a nonweighted case (which refers to $\lambda=\vone$). 
\begin{lem}\label{lem:finite}
Let $\M$ be an $R$-weighted mean $\M$ on $I$ and $\lambda \in W^0(R)$. Then $C= \Hc\M(\lambda)$ is the smallest extended real number such that
\Eq{E:finite}{
\sum_{n=1}^{N} \lambda_n \cdot \Mm_{i=1}^n\big(x_i,\lambda_i\big) \le C \cdot \sum_{n=1}^{N} \lambda_nx_n
\quad \text{for all $N \in \N$ and $(x_n)\in I^N$}.
}
\end{lem}
\begin{proof}
Fix $\lambda\in W^0(R)$.
In a limit case as $N \to \infty$, \eq{E:finite} implies \eq{E:WHI} with the same constant, thus $C \ge \Hc\M(\lambda)$. 
The remaining part is to verify \eq{E:finite} for $C=\Hc\M(\lambda)$. If $\Hc\M(\lambda)=\infty$, this inequality is trivially satisfied. Thus one can assume that $\M$ is a $\lambda$-weighted Hardy mean. 

Fix $(x_n) \in I^N$, $\varepsilon \in I$ and define
\Eq{*}{
y_n:=\begin{cases}
        x_n & \text{ for }n \le N\\
        \min(1,\lambda_n^{-1}) 2^{N-n}\varepsilon & \text{ for }n > N
       \end{cases}
}
Then we have
\Eq{*}{
\sum_{n=1}^{N} \lambda_n \cdot \Mm_{i=1}^n\big(x_i,\lambda_i\big) 
&=\sum_{n=1}^{N} \lambda_n \cdot \Mm_{i=1}^n\big(y_i,\lambda_i\big)
\le \sum_{n=1}^{\infty} \lambda_n \cdot \Mm_{i=1}^n\big(y_i,\lambda_i\big)
\le \Hc\M(\lambda) \cdot \sum_{n=1}^{\infty} \lambda_ny_n\\
&\le \Hc\M(\lambda) \cdot \Big( \sum_{n=N+1}^\infty 2^{N-n}\varepsilon+\sum_{n=1}^N \lambda_nx_n \Big)
=\Hc\M(\lambda) \cdot \Big( \varepsilon +\sum_{n=1}^N \lambda_nx_n\Big).
}
Now we can simply take $\varepsilon \to 0$ to obtain \eq{E:finite} with $C=\Hc\M(\lambda)$.
\end{proof}

Let us conclude this section with a characterization of the weighted Hardy property for the arithmetic mean. As a matter of fact, there are a substantial background beyond this result as the arithmetic mean is a boundary case in few contexts. First, it is the smallest power mean which does not admit the Hardy property (see the very beginning of this paper). Second, it is the largest concave mean, in particular all results related to Kedlaya inequality are stated for the means which are comparable to the arithmetic mean. Finally, the series which is related to the (nonweighted) Hardy property is divergent for every vector of nonnegative elements except the identically-zero sequence which has some further implications (cf. \cite{Pas1812}).

\begin{prop}\label{prop:Aryth}
Let $\A$ be the arithmetic mean and $\lambda \in W^0(\R)$. Then
\Eq{E:Aryth}{
\Hc\A(\lambda)=\sum_{m=1}^\infty \frac{\lambda_m}{\Lambda_m}.
}

In particular the arithmetic mean is a $\lambda$-Hardy mean if and only if $\sum_{n=1}^\infty \lambda_n<+\infty$.

\end{prop}
\begin{proof}
Take $x \in \ell^1(\lambda)$ arbitrarily. We have
\Eq{*}{
\sum_{n=1}^\infty \lambda_n \Ar_{k=1}^n(x_k,\lambda_k) 
=\sum_{n=1}^\infty \frac{\lambda_n}{\Lambda_n} \sum_{k=1}^n \lambda_kx_k
=\sum_{n=1}^\infty \Big( \sum_{m=n}^\infty \frac{\lambda_m}{\Lambda_m}\Big) \lambda_nx_n 
\le \sum_{m=1}^\infty \frac{\lambda_m}{\Lambda_m}\cdot \sum_{n=1}^\infty  \lambda_nx_n
}
Thus we obtain the $(\le)$ part of \eq{E:Aryth}. 
To prove the converse inequality fix $q \in (0,1)$ and take a sequence $x_n=\frac{q^n}{\lambda_n}$. Then $\sum_{n=1}^\infty \lambda_nx_n=\frac{q}{1-q}$. Thus
\Eq{*}{
\sum_{n=1}^\infty \lambda_n \Ar_{k=1}^n(x,\lambda) 
=\sum_{n=1}^\infty \Big(\sum_{m=n}^\infty \frac{\lambda_m}{\Lambda_m}\Big) q^n \ge \sum_{m=1}^\infty \frac{\lambda_m}{\Lambda_m} q=(1-q) \Big( \sum_{m=1}^\infty \frac{\lambda_m}{\Lambda_m} \Big)\sum_{n=1}^\infty \lambda_nx_n.
}
In  a limit case as $q \to 0$ we obtain the remaining inequality in \eq{E:Aryth}. Let us emphasize that this proof remains valid in the case $\sum_{m=1}^\infty \frac{\lambda_m}{\Lambda_m}=+\infty$. 
Finally, as the series $(\lambda_n)$ and $(\frac{\lambda_n}{\Lambda_n})$ are equiconvergent (see Lemma~\ref{lem:equicon}) we obtain the moreover part.
\end{proof}

\section{Main result}

In what follows we are heading towards the sufficient condition for $\M$ and $\lambda$ to validate the equality $\Hc[\lambda]{} = \Hc[\vone]{}$. In view of Theorem~\ref{thm:mu1} the $(\le)$ inequality is satisfied for all symmetric, monotone means and all vectors $\lambda$. Therefore we need to show the converse inequality. 
The idea is similar to the one which was used in \cite[section~5]{PalPas19b}.

First we generalize Lemma~\ref{lem:L5.2} by replacing $\vone$ by a vector $\lambda$ satisfying certain properties.

\begin{lem}
\label{lem:1gamma}
Let $\M$ be a symmetric and monotone $R^*$-weighted mean on $I$. Let $\psi \in W^0(R^*)$ and $\lambda \in W^0(R^*)$ with $\Lambda_n \to \infty$ and $\lambda_n/\Lambda_n\to 0$. Then the inequality 
 \Eq{E:1gamma}{
\sum_{m=1}^M \psi_m \Mm_{i=1}^m \big(y_i,\psi_i\big)&\le \Hc\M(\lambda) \sum_{m=1}^M \psi_m y_m.
} 
is valid for every $M \in \N$ and every nonincreasing sequence $y \in I^M$.
\end{lem}
Its technical and quite lengthy proof is shifted to section~\ref{sec:1gamma}. As a direct consequence, using some already known results, we can prove our next theorem. It is inspired by a proof of Theorem~\ref{thm:mu1}.

\begin{thm}\label{thm:main}
For every symmetric, monotone $R$-weighted mean $\M$ and a vector $\lambda \in W^0(R)$ such that $\Lambda_n\to \infty$ and $\lambda_n/\Lambda_n\to 0$, the equality $\Hc\M(\vone)=\Hc\M(\lambda)$ is valid.
\end{thm}
\begin{proof}
Let $N \in \N$ and $x \in I^N$. First, by Proposition~\ref{prop:fieldex} we may extend $\M$ to the \mbox{$R^*$-weighted} mean. 
Next, by Lemma~\ref{lem:rearrangement} there exists $M\in \N$, a nonincreasing sequence $y \in I^M$ and $\psi \in W^0_M(R^*)$ such that $\sum_{n=1}^N x_n=\sum_{n=1}^M \psi_ny_n$ and 
\Eq{*}{
\sum_{n=1}^N \Mm_{i=1}^n \big(x_i,1\big)&\le \sum_{m=1}^M \psi_m \Mm_{i=1}^m \big(y_i,\psi_i\big).
}
Now, applying Lemma~\ref{lem:1gamma} we obtain
\Eq{*}{
\sum_{n=1}^N \Mm_{i=1}^n \big(x_i,1\big) \le \sum_{m=1}^M \psi_m \Mm_{i=1}^m \big(y_i,\psi_i\big) \le \Hc\M(\lambda) \sum_{m=1}^{M} \psi_my_m = \Hc\M(\lambda) \sum_{n=1}^{N} x_n.
}

Finally, by Lemma~\ref{lem:finite} we get $\Hc\M(\vone) \le \Hc\M(\lambda)$. This ends the proof as the converse inequality is a direct implication of Theorem~\ref{thm:mu1}.
\end{proof}

\subsection{Partition ordering and cut theorem} 
In this section we intend to show some monotonicity-type result for Hardy constant. First let us introduce some preorder on vector of real numers

\begin{defin}[Partition ordering]
 We define the order $\prec$ on infinite sequences of real numbers in the following way:
 $(\alpha_k)_{k=1}^\infty \prec (\beta_n)_{n=1}^{\infty}$ if there exists a nondecreasing, divergent sequence $(n_k)_{k=0}^\infty$ with $n_0=0$ such that
 $\alpha_k=\sum\limits_{n=n_{k-1}+1}^{n_k} \beta_n$ (for $n_k=n_{k-1}$ we assume $\alpha_k=0$).
 \end{defin}

It can be shown that if we restrict our consideration to a vectors with all positive entries then $\prec$ is the partial order. As a matter of fact, this order is related to Hardy constant
 
\begin{thm}[Cut theorem]\label{thm:cut}
 Let $\M$ be a monotone and Jensen concave
$\Q$-weighted mean on $I$ (resp. $\R$-weighted mean on $I$ which is continuous in its weights). The mapping 
$\Hc\M$ is monotone with respect to $\prec$. More precisely for every $\psi,\lambda \in W^0(R)$ with $\psi \prec \lambda$ we have $\Hc\M(\psi) \le \Hc\M(\lambda)$ (here $R =\Q$ or $R=\R$ depending on the context).
\end{thm}
\begin{proof}
Take $M \in \N$ and $x \in I^M$. By Theorem~\ref{thm:samesum}, there exists nonincreasing sequence $y\in I^M$ such that $\sum_{m=1}^M \psi_m x_m = \sum_{m=1}^M \psi_m y_m$ and 
 \Eq{*}{
\Mm_{i=1}^m (x_i,\psi_i) \le \Mm_{i=1}^m (y_i,\psi_i) \text{ for all }m \in\{1,\dots,M\}.
}

With the usual notation $\Lambda_n=\lambda_1+\dots+\lambda_n$ and $\Psi_m:=\psi_1+\dots+\psi_m$ (with $\Psi_0=\Lambda_0=0$), by $\psi \prec \lambda$ there exists a sequence $(n_m)_{m=1}^\infty$ such that $\Psi_m=\Lambda_{n_m}$ for all $m \in\N_+ \cup \{0\}$. Denote briefly $N:=n_M$, i.e. $\Psi_M=\Lambda_N$.
Using all these facts jointly with Lemma~\ref{lem:decr} we get
\Eq{*}{
\sum_{m=1}^M \psi_m \Mm_{i=1}^m (y_i,\psi_i) 
&= \sum_{m=1}^M \psi_m \Mm_{0}^{\Psi_m} \chi_{y,\psi}(t) \:dt
=\sum_{m=1}^M \sum_{n=n_{m-1}+1}^{n_m}\lambda_n \Mm_0^{\Lambda_{n_m}} \chi_{y,\psi}(t) \:dt\\
&\le
\sum_{m=1}^M \sum_{n=n_{m-1}+1}^{n_m}\lambda_n \Mm_0^{\Lambda_n} \chi_{y,\psi}(t) \:dt
=\sum_{n=1}^{N}\lambda_n \Mm_0^{\Lambda_{n}} \chi_{y,\psi}(t) \:dt
}
Let us now observe that $\chi_{y,\psi}$ is constant on every interval $[\Lambda_{i-1},\Lambda_i)$. Therefore by Lemma~\ref{lem:finite} we have 
\Eq{*}{
\sum_{n=1}^{N}\lambda_n \Mm_0^{\Lambda_{n}} \chi_{y,\psi}(t) \:dt
=\sum_{n=1}^{N}\lambda_n \Mm_{i=1}^n\big(\chi_{y,\psi}(\Lambda_{i-1}),\lambda_i\big) 
&\le \Hc\M(\lambda) \sum_{n=1}^{N}\lambda_n\chi_{y,\psi}(\Lambda_{n-1}).
}
But
\Eq{*}{
\sum_{n=1}^{N}\lambda_n\chi_{y,\psi}(\Lambda_{n-1})
&= \sum_{n=1}^{N}\int_{\Lambda_{n-1}}^{\Lambda_n}\chi_{y,\psi}(t)\:dt
=\int_0^{\Lambda_N}\chi_{y,\psi}(t)\:dt\\
&=\int_0^{\Psi_M}\chi_{y,\psi}(t)\:dt
=\sum_{m=1}^M \psi_my_m
=\sum_{m=1}^M \psi_mx_m.
}
Binding all properties above we get
\Eq{*}{
\sum_{m=1}^M \psi_m \Mm_{i=1}^m (x_i,\psi_i)
&\le \sum_{m=1}^M \psi_m \Mm_{i=1}^m (y_i,\psi_i)
\le \sum_{n=1}^N \lambda_n \Mm_0^{\Lambda_{n}} \chi_{y,\psi}(t) \:dt \\
&\le \Hc\M(\lambda) \sum_{n=1}^N\lambda_n\chi_{y,\psi}(\Lambda_{n-1})= \Hc\M(\lambda) \sum_{m=1}^M \psi_mx_m.
}
Finally, as $M \in \N$ was taken arbitrarily by Lemma~\ref{lem:finite} we obtain $\Hc\M(\psi)\le \Hc\M(\lambda)$.
\end{proof}

Let us now present a simple application of this result.
\begin{cor}
 Let $\M$ be a monotone and Jensen concave
$\Q$-weighted mean on $I$ (resp. $\R$-weighted mean on $I$ which is continuous in its weights). Let 
$C\colon (0,\infty) \to [1,+\infty]$ be given by
\Eq{*}{
C(q):= \Hc\M\big(\big(q^n\big)_{n=1}^\infty\big).
}
Then $C(q^r) \le C(q)$ for all $q \in (0,\infty)$ and $r \in \N$.
\end{cor}
\begin{proof}
For $q=1$ this statement is trivial. For $q \in (0,\infty) \setminus \{1\}$ define two vectors 
\Eq{*}{
\lambda:=\big(q^n\big)_{n=1}^\infty
\quad \text{ and }\quad \psi=\Big(\dfrac{1-q^r}{1-q} q^{rk}\Big)_{k=1}^\infty\:.
}
First we prove that $\psi \prec \lambda$. Indeed, for a sequence $(n_k)_{k=1}^\infty=(r \cdot k)_{k=1}^\infty$ we have
\Eq{*}{
\sum_{n=n_k}^{n_{k+1}-1} \lambda_n=\sum_{n=rk}^{r(k+1)-1} q^n
=q^{rk} \sum_{n=0}^{r-1} q^n=q^{rk} \frac{1-q^r}{1-q}=\psi_k.
}
 Therefore, by nullhomogeneity in weights and Theorem~\ref{thm:cut} we have 
 \Eq{*}{
C(q^r)=\Hc\M\big(\big(q^{rk}\big)_{k=1}^\infty\big)=\Hc\M(\psi) 
\le \Hc\M(\lambda)=C(q),
 }
which is the statement.
 \end{proof}

\subsection{Lower semicontinuouity} Next results show that for every mean a Hardy constant is lower semicontinuous as a function of weight sequence (in a pointwise topology).

\begin{defin}[Pointwise topology] Let $\lambda,\psi^{(1)},\psi^{(2)},\dots$ be elements in $\R^\N$. We say that the sequence $(\psi^{(k)})_{k=1}^\infty$ converges to $\lambda$ in \emph{pointwise topology} if $\lim_{k \to \infty} \psi_n^{(k)}=\lambda_n$ for all $n \in \N$. We denote it brifely by $\psi^{(k)} \pto \lambda$.
\end{defin}
Whenever the sequence $\lambda$ contains only positive terms we can rewrite this in an equvalient form: for all $\theta<1$ and $N \in \N$ there exists $k_0\in \N$ such that
\Eq{*}{
\abs{\frac{\psi^{(k)}_n}{\lambda_n}}\in (\theta,\theta^{-1}) \text{ for all }n \in \{1,\dots,N\}\text{ and }k \ge k_0.
}

Main result of this section states as follows

\begin{thm}\label{thm:semicontinuous}
For every weighted mean $\M$ which is continuous it its weights the mapping $\Hc\M$ is lower semicontinuous in the pointwise topology.
\end{thm}
\begin{proof}
Take a sequence $(\psi^{(k)})_{k=1}^\infty$ of elements in $W^0(R)$ which is convergent to $\lambda \in W^0(R)$ in the pointwise topology.
We prove that 
\Eq{E:LS}{
\liminf_{k \to \infty} \Hc\M\big(\psi^{(k)}\big) \ge \Hc\M(\lambda).} 

Fix $\theta\in (0,1)$. There exists a sequence $x \in \ell^1(\lambda)$ such that
 \Eq{*}{
 \sum_{n=1}^{\infty} \lambda_n \Mm_{i=1}^n (x_i,\lambda_i) > \theta \Hc\M(\lambda) \sum_{n=1}^{\infty} \lambda_n x_n 
 }
As the series on the left hand side is convergent we can take $N_\theta \in \N$ with
\Eq{*}{
 \sum_{n=1}^{N_\theta} \lambda_n \Mm_{i=1}^n (x_i,\lambda_i) &\ge \theta \sum_{n=1}^{\infty} \lambda_n \Mm_{i=1}^n (x_i,\lambda_i).
 }
Once $N_\theta$ is fixed, $\psi^{(k)} \pto \lambda$, and $\M$ is continuous in its weights, there exists $k_\theta\in \N$ such that
\Eq{*}{
\psi^{(k)}_n \Mm_{i=1}^n (x_i,\psi^{(k)}_i) \ge \theta\lambda_n \Mm_{i=1}^n (x_i,\lambda_i) \text{ for all }n \le N_\theta \text{ and }k \ge k_\theta.
}
Then for all $k \ge k_\theta$ we have
\Eq{*}{
\sum_{n=1}^\infty \psi^{(k)}_n \Mm_{i=1}^n (x_i,\psi^{(k)}_i)
&>
\sum_{n=1}^{N_\theta} \psi^{(k)}_n \Mm_{i=1}^n (x_i,\psi^{(k)}_i)
\ge \theta  \sum_{n=1}^{N_\theta} \lambda_n \Mm_{i=1}^n (x_i,\lambda_i)\\
&\ge \theta^2 \sum_{n=1}^{\infty} \lambda_n \Mm_{i=1}^n (x_i,\lambda_i)
> \theta^3 \Hc\M(\lambda) \sum_{k=1}^{\infty} \lambda_i x_i.
}
Thus $\Hc\M\big(\psi^{(k)}\big) > \theta^3 \Hc\M(\lambda)$ for all $k \ge k_\theta$. As $\theta \in (0,1)$ was taken arbitrarily we obtain the inequality \eq{E:LS}.
\end{proof}

Let us now show that the mapping mentioned in the theorem above is not necessarily  continuous.
\begin{exa}\label{ex:AR}
Define $\lambda,\psi^{(1)},\psi^{(2)},\dots \in W^0(\R)$ by
\Eq{*}{
\lambda_i &=\tfrac{1}{2^i} \quad(i \in \N),\\
\psi^{(k)}_i &=\begin{cases} \frac{1}{2^i} & i \in \N \setminus \{k\} \\
1 & i=k \end{cases} \qquad  \text{ for all }k \in \N.
}
Obviously $\psi^{(k)} \pto \lambda$ and $\Lambda_n=1-\tfrac1{2^n}$ for all $n \in\N$. 
Then Proposition~\ref{prop:Aryth} implies 
\Eq{*}{
\Hc\A(\lambda)=\sum_{m=1}^\infty \frac{\lambda_m}{\Lambda_m}=\sum_{m=1}^\infty \frac{1}{2^m-1}=:E,
}
where $E\approx1.606$ is so-called Erd\"os–Borwein constant.
On the other hand, for all $k\in \N$, 
\Eq{*}{
\Hc\A\big(\psi^{(k)}\big)=\sum_{m=1}^\infty \frac{\psi^{(k)}_m}{\Psi^{(k)}_m} = \sum_{m=1}^{k-1} \frac{\psi^{(k)}_m}{\Psi^{(k)}_m}+\frac{\psi^{(k)}_k}{\Psi^{(k)}_k}+\sum_{m={k+1}}^{\infty} \frac{\psi^{(k)}_m}{\Psi^{(k)}_m}.
}
If we take the limit as $k \to \infty$ we get 
\Eq{*}{
\lim_{k\to \infty}\sum_{m=1}^{k-1} \frac{\psi^{(k)}_m}{\Psi^{(k)}_m}&=\lim_{k\to \infty}\sum_{m=1}^{k-1} \frac{\lambda_m}{\Lambda_m}=E,\\
\lim_{k\to \infty}\frac{\psi^{(k)}_k}{\Psi^{(k)}_k}&=
\lim_{k\to \infty}\frac{1}{\Lambda_{k-1}+1}=
\lim_{k\to \infty}\frac{1}{2-\frac1{2^{k-1}}}=\frac12,\\
(0\le)\:\: \lim_{k\to \infty}\sum_{m={k+1}}^{\infty} \frac{\psi^{(k)}_m}{\Psi^{(k)}_m}&\le\lim_{k\to \infty}\sum_{m={k+1}}^{\infty} \frac{1}{2^{k+1}}=0.
}
Thus 
\Eq{*}{
\lim_{k \to \infty} \Hc\A\big(\psi^{(k)}\big)=E+\tfrac{1}{2}>\Hc\A(\lambda).
}
This shows that inequality \eq{E:LS} can be strict and consequently that the mapping $\Hc{\M}$ is not continuous (for $\M=\A$).
\end{exa}

Let us now present an important application of Theorem~\ref{thm:semicontinuous}.

\begin{cor}\label{cor:1Continuous}
Let $\M$ be a symmetric and monotone $R$-weighted mean on $I$ which is continuous in its weights. For every sequence $(\psi^{(k)})$ of elements in $W^0(R)$ with $\psi^{(k)} \pto \vone$ we have
$\Hc\M\big(\psi^{(k)}\big) \to \Hc[\vone]{\M}$.
In particular
\Eq{*}{
\lim_{s \to 1}\Hc{\M}\big((s^n)_{n=1}^\infty\big) = \Hc[\vone]{\M}.
}
\end{cor}

Its proof is straightforward in view of Theorems~\ref{thm:mu1} and \ref{thm:semicontinuous}. 
This statement is related to \cite[Theorem 5.5]{PalPas20}, where such Hardy constants were obtained for a concave quasideviation means (under some additional assumptions). 

Let us conlude this section with a natural open problem. It was shown in Theorem~\ref{thm:semicontinuous} that $\Hc\M$ is lower semicontinuous. By Example~\ref{ex:AR}, we know that it is not continuous in a case $\M=\A$. However we suppose that it is the case for Hardy ($\vone$-Hardy) means.
\begin{conj}
 Let $\M$ be a symmetric and monotone $R$-weighted mean on $I$  which is continuous in its weights. If $\M$ is a $\vone$-Hardy mean (equivalently it is a $\lambda$-Hardy mean for all $\lambda \in W^0(R)$) then $\Hc\M$ is continuous in the pointwise topology. 
\end{conj}

\section{Applications}

We now aim to present few implications of the latter results.

\subsection{Gini means}
First we recall a clasical notion of Gini means in a weighted setting \cite{Gin38}. For $p,q \in \R$ the Gini mean $\G_{p,q}$ is the $\R$-weighted mean on $\R_+$ defined by
Another important generalization of power means was introduced in the paper \cite{Gin38}. For two real parameters $p,q$, the Gini mean $\G_{p,q}\colon \bigcup_{n=1}^{\infty} \R_+^n \times W_n(\R) \to \R_+$ is defined by
\Eq{*}{
\G_{p,q} (x,\lambda):= 
\begin{cases} 
\left(\dfrac{\lambda_1x_1^p+\cdots+\lambda_nx_n^p}{\lambda_1x_1^q+\cdots+\lambda_nx_n^q} \right)^{\frac{1}{p-q}} &\quad \text{ if } p \ne q, \\[4mm]                                                         
\exp\left(\dfrac{\lambda_1x_1^p\log x_1+\cdots+\lambda_nx_n^p\log x_n}{\lambda_1x_1^p+\cdots+\lambda_nx_n^p} \right) &\quad \text{ if } p = q.
\end{cases}
}

In a sequence papers, further generalizations were obtained: \emph{Bajraktarevi\'c means} \cite{Baj58}, \emph{deviation (or Dar\'oczy) means} \cite{Dar71b} and \emph{quasideviation means} \cite{Pal82a}. For more details, we just refer the reader to a series of papers by Losonczi \cite{Los70a,Los71a,Los71b,Los71c,Los73a,Los77} (for Bajraktarevi\'c means), Dar\'oczy \cite{Dar71b,Dar72b}, Dar\'oczy--Losonczi \cite{DarLos70}, Dar\'oczy--P\'ales \cite{DarPal82,DarPal83} (for deviation means), P\'ales \cite{Pal82a,Pal83b,Pal84a,Pal85a,Pal88a,Pal88d,Pal88e} (for deviation and quasideviation means) and P\'ales--Pasteczka \cite{PalPas19b} (for semideviation means).

Clearly, in the particular case $q=0$, the mean $\G_{p,q}$ reduces to the $p$th power mean $\P_p$. It is also obvious that $\G_{p,q}=\G_{q,p}$.
It is known \cite{Los71a,Los71c} that $\G_{p,q}$ is monotone if and only if $pq\le0$ and  concave if and only if 
\Eq{pq}{
\min(p,q)\leq0\leq\max(p,q)\leq 1.
}
Furthermore the following properties involving Hardy inequality are valid.
\begin{lem}[\cite{PalPas20}, Proposition~5.2]
\label{lem:Gini}
Let $(\lambda_n)\in W_0$ such that $\Lambda_n\to\infty$ and $\big(\tfrac{\lambda_n}{\Lambda_n}\big)_{n=1}^{\infty}$ is nonincreasing with a limit $\eta\in[0,1)$. Let $p,\,q \in \R$ satisfying \eq{pq}. Then 
\Eq{*}{
\Hc[\lambda]{\G_{p,q}}=
\begin{cases}
\bigg(\dfrac{1-(1-\eta)^{1-q}}{1-(1-\eta)^{1-p}}\bigg)^{\tfrac{1}{p-q}}, &\quad \eta \in (0,1) \text{ and } p\ne q;\\[1mm]
\bigg(\dfrac{1-q}{1-p}\bigg)^{\tfrac{1}{p-q}}, &\quad \eta = 0 \text{ and } p\ne q; \\[4mm]
(1-\eta)^{1-1/\eta}, &\quad \eta \in (0,1) \text{ and } p=q= 0;\\[1mm]
e, &\quad \eta = 0 \text{ and } p=q=0.
\end{cases}
}
\end{lem}

Now we show that the monotonicity assumption of the  ratio sequence $\big(\tfrac{\lambda_n}{\Lambda_n}\big)_{n=1}^{\infty}$ can be ommited in the case $\eta=0$.
\begin{prop}\label{prop:G0}
Let $(\lambda_n)\in W_0$ such that $\Lambda_n\to\infty$ and $\tfrac{\lambda_n}{\Lambda_n}\to 0$. Let $p,\,q \in \R$, $\min(p,q)\le 0 \le \max(p,q)<1$. Then 
\Eq{E:Gpqmu0}{
\Hc[\lambda]{\G_{p,q}}=
\begin{cases}
\bigg(\dfrac{1-q}{1-p}\bigg)^{\tfrac{1}{p-q}}, &\quad p\ne q; \\[4mm]
e, &\quad p=q=0.
\end{cases}
} 
\end{prop}
\begin{proof}
 As $\G_{p,q}$ is monotone by Theorem~\ref{thm:main} we have $\Hc[\lambda] {\G_{p,q}}=\Hc[\vone] {\G_{p,q}}$. Now, using 
 Lemma~\ref{lem:Gini} to a vector $\lambda=\vone$ we obtain desired equality \eq{E:Gpqmu0}.
\end{proof}

\begin{rem}
 We can analogously relax such a monotonicity assumption for all concave quasideviation means -- see \cite{PalPas19a,PalPas20} for detailed study of their Hardy property.
\end{rem}

Now we establish completely new Hardy-type inequality related to Gini means

\begin{prop}\label{prop:pqtau}
For all $p,q \in \R$ satisfying \eq{pq}, $\tau \in \R$, and $x \in \ell_1(\lambda)$ we have 
\Eq{tau}{
 \sum_{n=1}^\infty \lambda_n x_n^\tau \Big(\G_{(1-\tau)p+\tau,(1-\tau)q+\tau} \big((x_1,\dots,x_n),(\lambda_1,\dots,\lambda_n)\big)\Big)^{1-\tau}
 \le \Hc{\G_{p,q}}(\vone) \sum_{n=1}^\infty \lambda_n x_n.
 } 
\end{prop}

\begin{proof}
 Fix $p,\,q \in \R$ satisfying \eq{pq} and $d \in \R$. Applying Proposition~\ref{prop:G0} to a sequence of weights with $\lambda_n$ replaced by $\mu_n:=\lambda_ny_n^d$ we obtain
 \Eq{*}{
 \sum_{n=1}^\infty \lambda_n y_n^d \G_{p,q} \big((y_1,\dots,y_n),(\lambda_1y_1^d,\dots,\lambda_ny_n^d)\big)
 \le \Hc{\G_{p,q}}(\vone) \sum_{n=1}^\infty \lambda_n y_n^{d+1}
 }
However by the definition of Gini mean we easily get 
 \Eq{*}{
\G_{p,q} \big((y_1,\dots,y_n),(\lambda_1y_1^d,\dots,\lambda_ny_n^d)\big)=
\Big(\G_{\frac{p+d}{d+1},\frac{q+d}{d+1}} \big((y_1^{d+1},\dots,y_n^{d+1}),(\lambda_1,\dots,\lambda_n)\big)\Big)^{\frac{1}{d+1}}.
 }
Binding above results and putting $x_n:=y_n^{d+1}$ we get 
 \Eq{*}{
 \sum_{n=1}^\infty \lambda_n x_n^\frac{d}{d+1} \Big(\G_{\frac{p+d}{d+1},\frac{q+d}{d+1}} \big((x_1,\dots,x_n),(\lambda_1,\dots,\lambda_n)\big)\Big)^{\frac{1}{d+1}}
 \le \Hc{\G_{p,q}}(\vone) \sum_{n=1}^\infty \lambda_n x_n.
 }
 Now we can put $d=\frac{\tau}{1-\tau}$ 
 to obtain desired inequality \eq{tau}.
 \end{proof}

 We can now use above resuly to give some majorization of the Hardy constant for Gini means in a rectangle $(-\infty,0)^2$.
 \begin{thm}\label{thm:Ginineg}
Let $p,q \in (-\infty,0)$ with $p\ge q$. Then
\Eq{Gneg}{
\Hc[\lambda]{\G_{p,q}}\le \frac{C\big(\frac{q-p}{1-p}\big)-p}{1-p}\qquad \text{ for all }\lambda \in W^0(\R),
}
where $C(p)$ was defined in \eq{E:C(p)}.
 \end{thm}
\begin{proof}
 Let $q_0:=\frac{q-p}{1-p} \in (-\infty,0)$. Fix $\lambda\in W^0(\R)$ and $x\in \ell_1(\lambda)$.
 Applying Proposition~\ref{prop:pqtau} with $(p,q,\tau)\leftarrow (0,\frac{q-p}{1-p},p)$ inequality \eq{tau} becomes
 \Eq{*}{
  \sum_{n=1}^\infty \lambda_n x_n^p \Big(\G_{p,(1-p)\frac{q-p}{1-p}+p} \big((x_1,\dots,x_n),(\lambda_1,\dots,\lambda_n)\big)\Big)^{1-p}
 \le \Hc{\G_{0,\frac{q-p}{1-p}}}(\vone) \sum_{n=1}^\infty \lambda_n x_n.
 }
 Thus, by the identities $(1-p)\frac{q-p}{1-p}+p=q$, $\G_{0,\frac{q-p}{1-p}}=\P_{\frac{q-p}{1-p}}$, and \eq{E:C(p)0} we get
 \Eq{*}{
  \sum_{n=1}^\infty \lambda_n x_n^p \Big(\G_{p,q} \big((x_1,\dots,x_n),(\lambda_1,\dots,\lambda_n)\big)\Big)^{1-p}
 \le C\big(\tfrac{q-p}{1-p}\big) \sum_{n=1}^\infty \lambda_n x_n.
 }
Applying the inverse Cauchy inequality $pu+(1-p)v \le u^p v^{1-p}$ which is valid for all $u,v \in(0,+\infty)$ and $p<0$ we obtain
 \Eq{*}{
  p \sum_{n=1}^\infty \lambda_n x_n + (1-p)\sum_{n=1}^\infty \lambda_n \G_{p,q} \big((x_1,\dots,x_n),(\lambda_1,\dots,\lambda_n)\big)
 \le C\big(\tfrac{q-p}{1-p}\big) \sum_{n=1}^\infty \lambda_n x_n,
 }
 which is trivially equivalent to
 \Eq{*}{
 \sum_{n=1}^\infty \lambda_n \G_{p,q} \big((x_1,\dots,x_n),(\lambda_1,\dots,\lambda_n)\big)\le \frac{C\big(\tfrac{q-p}{1-p}\big)-p}{1-p}\sum_{n=1}^\infty \lambda_n x_n.
 }
 The latter statement implies \eq{Gneg}
\end{proof}

%
%

\begin{rem}
Using some homogenization techniques described in \cite{PalPas19a} one can find anouther estmation for the Hardy constant of Gimi means in a case $p,q \in (-\infty,0)$ with $p\ne q$. More precisely one we can show that
$\Hc[\lambda]{\G_{p,q}}\le \xi_{p,q}$, where $\xi_{p,q}$ is the unique solution $\xi\in (1,e)$ of the equation
\Eq{*}{
\frac{p}{p-1}\Big(\frac qp\Big)^{\frac{p-1}{p-q}}-\frac q{q-1}\Big(\frac qp\Big)^{\frac{q-1}{p-q}}=\frac1{1-q}\xi^{1-q}-\frac1{1-p}\xi^{1-p}.
}
This result has also the relevant limit counterpart when $p =q$. Approximate calculations motivates us to conjecture that this estimate is better than the one given in Theorem~\ref{thm:Ginineg}.
\end{rem}

\subsection{Repeted sequences} For $\lambda \in W^0(R)$ and $k \in \N$ define $\lambda^{\odot n} \in W^0(R)$ by
\Eq{*}{
(\lambda^{\odot n})_q=\lambda_{\ceil{q/n}} \qquad \text{ for all }q \in\N.
}
Intuitively $\lambda^{\odot n}$ is obtained by repeating $n$-times each element of $\lambda$. Such notation is motivated by the identity $\lambda \odot \lambda = \lambda^{\odot 2}$ with the two-parameter operator $\odot$ which already appeared (for finite vectors) in Definition~\ref{def:WQA}. Obviously $\lambda^{\odot n}=\lambda$ precisely if $n=1$ or $\lambda$ is a constant vector. 
We show the mapping $n \mapsto \Hc{\M}(\lambda^{\odot n})$ is monotone with the division (partial) ordering for a vast family of means.


\begin{prop}
Let $\M$ be a monotone and Jensen concave
$\Q$-weighted mean on $I$ (resp. $\R$-weighted mean on $I$ which is continuous in its weights) and $\lambda\in W^0(\Q)$ (resp. $\lambda\in W^0(\R)$). Then for all $m, n \in \N$ with $\tfrac nm\in\N$ we have $\Hc\M(\lambda^{\odot m}) \le \Hc\M(\lambda^{\odot n})$. Moreover 
\Eq{YYDF}{
\lim_{n \to \infty} \Hc\M(\lambda^{\odot n})=\Hc\M(\vone).
}
\end{prop}
\begin{proof}
 Let $k:=\frac{n}{m}$. Then we have 
 \Eq{*}{
 \sum_{i=1}^k (\lambda^{\odot n})_{ks+i}= \sum_{i=1}^k \lambda_{\ceil{\frac{ks+i}n}}=\sum_{i=1}^k \lambda_{\ceil{\frac{ks+i}{km}}}
\text{ for all }s \in \N. }
 Now observe that $\tfrac{ks+i}{km} \notin \N$ for $i\in\{1,\dots,k-1\}$. Thus 
$\ceil{\tfrac{ks+i}{km}}=\ceil{\tfrac{ks+k}{km}}=\ceil{\tfrac{s+1}{m}}$ for all $i\in\{1,\dots,k\}$ and, consequently,
\Eq{*}{
\sum_{i=1}^k (\lambda^{\odot n})_{ks+i}=k\lambda_{\ceil{\tfrac{s+1}{m}}}=k (\lambda^{\odot m})_{s}.
}
Therefore $(k\lambda^{\odot m})\prec (\lambda^{\odot n})$. Now by Theorem~\ref{thm:cut} and nullhomogeneity in the weights we obtain 
$\Hc\M(\lambda^{\odot m}) \le \Hc\M(k\lambda^{\odot n})=\Hc\M(\lambda^{\odot n})$. To show the moreover part observe that $(\lambda^{\odot n})\pto \lambda_1\vone$ and therefore by  nullhomogeneity in weights and Corollary~\ref{cor:1Continuous} we obtain the equality  
$\lim\limits_{n \to \infty} \Hc\M(\lambda^{\odot n})=\Hc\M(\lambda_1\vone)=\Hc\M(\vone)$, which is \eq{YYDF}.
\end{proof}

\section{\label{sec:1gamma} Proof of Lemma~\ref{lem:1gamma}}

In this section we reuse the concepts which were contained in the proof of \cite[Lemma~5.2]{PalPas19b} (see Lemma~\ref{lem:L5.2} above). 

First observe that if $\Hc\M(\lambda)=+\infty$ then \eq{E:1gamma} is trivially valid. From now on assume that $\Hc\M(\lambda)<+\infty$. In order to make the proofs more compact, define $\Psi_m:=\psi_1+\cdots+\psi_m$ for $m \in\{1,\dots,M\}$ and $\Lambda_n:=\lambda_1+\dots+\lambda_n$ for $n\in\N$. In view of the nullhomogeneity of $\M$, we may assume that $\Psi_M=1$. For each $j\in\N$ define the $R^*$-simple function $f_j \colon [0,1) \to I$ by
\Eq{*}{
f_j\big|_{ \mbox{$\big[\frac{\Lambda_s}{\Lambda_j},\,\frac{\Lambda_{s+1}}{\Lambda_j}\big)$}}:=\chi_{y,\psi}\big(\tfrac{\Lambda_s}{\Lambda_j}\big) \quad \text{ for all } s\in\{0,\dots,j-1\}.
}
As the sequence $y$ is nonincreasing, $\chi_{y,\psi}$ is nonincreasing, too. Therefore, for all $j \in\N$, the function $f_j$ is nonincreasing and $\chi_{y,\psi} \le f_j$. Thus, by Lemma~\ref{lem:decr}, so is the function $C_j \colon [0,1) \to I$ given by
\Eq{*}{
C_j(t):=\begin{cases}
         \hskip5mm y_1 &\mbox{ if }t=0,\\
         \inf\limits_{s\in [0, t] \cap R^*}\:\Mm\limits_0^s f_j(x) dx &\mbox{ if }t\in(0,1), 
        \end{cases} \qquad (j\in\N).
}
Fix $j \in \N$ with $j \ge 2$ arbitrarily. As $C_j$ is monotone, it is also Riemann integrable. Whence, for all $m \in\{1,\dots,M\}$ we get
\Eq{*}{
\psi_m \cdot \Mm_{i=1}^m \big(y_i,\psi_i\big) &= \psi_m \cdot \Mm_0^{\Psi_m} \chi_{y,\psi}(x) dx \le \psi_m \cdot \Mm_0^{\Psi_m} f_j(x) dx\\
&= \psi_m \cdot C_j(\Psi_m)= \int_{\Psi_{m-1}}^{\Psi_m} C_j(\Psi_m) dx\le \int_{\Psi_{m-1}}^{\Psi_m} C_j(x) dx.
}
Therefore, if we sum-up these inequalities side-by-side, we obtain
\Eq{Fin1}{
\sum_{m=1}^{M} \psi_m \cdot \Mm_{i=1}^m \big(y_i,\psi_i\big) 
&\le \int_{0}^1 C_j(x) dx.
}

We are now going to majorize the right hand side of this inequality. Observe first that
\Eq{SP1}{
\int_{0}^{\Lambda_1/\Lambda_j} C_j(x)dx \le \frac{\Lambda_1}{\Lambda_j}\cdot C_j(0) = \frac{\lambda_1y_1}{\Lambda_j}.
}
Furthermore, for all $n \in \{1,\dots,j-1\}$, as $C_j$ is nonincreasing and $\frac{\Lambda_n}{\Lambda_j} \in R^*$ we get
\Eq{SP2}{
\int_{\frac{\Lambda_n}{\Lambda_j}}^{\frac{\Lambda_{n+1}}{\Lambda_j}} C_j(x)dx 
\le \frac{\lambda_{n+1}}{\Lambda_j} C_j\Big(\frac{\Lambda_n}{\Lambda_j} \Big)
=\frac{\lambda_{n+1}}{\Lambda_j}\cdot \Mm_0^{\Lambda_n/\Lambda_j} f_j(x) dx
=\frac{\lambda_{n+1}}{\Lambda_j}\cdot \Mm_{i=1}^{n} \Big(\chi_{y,\psi} \big(\tfrac{\Lambda_{i-1}}{\Lambda_j}\big),\lambda_i\Big).
}
If we now sum up \eq{SP1} and \eq{SP2} for all $n\in\{1,\dots,j-1\}$, we obtain
\Eq{*}{
\int_{0}^1 C_j(x) dx 
&=\int_0^{\frac{\Lambda_1}{\Lambda_j}} C_j(x)dx
+\sum_{n=1}^{j-1} \int_{\frac{\Lambda_n}{\Lambda_j}}^{\frac{\Lambda_{n+1}}{\Lambda_j}} C_j(x)dx
\le \frac{\lambda_1}{\Lambda_j} y_1 + \sum_{n=1}^{j-1} \frac{\lambda_{n+1}}{\Lambda_j}\cdot \Mm_{i=1}^{n} \Big(\chi_{y,\psi} \big(\tfrac{\Lambda_{i-1}}{\Lambda_j}\big),\lambda_i\Big)\\
&=\frac{\lambda_1}{\Lambda_j} y_1 + \sum_{n=1}^{j-1} \frac{\lambda_n}{\Lambda_j}\cdot \Mm_{i=1}^{n} \Big(\chi_{y,\psi} \big(\tfrac{\Lambda_{i-1}}{\Lambda_j}\big),\lambda_i\Big)
+ \sum_{n=1}^{j-1} \frac{\lambda_{n+1}-\lambda_n}{\Lambda_j}\cdot \Mm_{i=1}^{n} \Big(\chi_{y,\psi} \big(\tfrac{\Lambda_{i-1}}{\Lambda_j}\big),\lambda_i\Big).
}
Now, by Lemma~\ref{lem:finite},
\Eq{*}{
\sum_{n=1}^{j-1} \frac{\lambda_n}{\Lambda_j}\cdot \Mm_{i=1}^{n} \Big(f_j \big(\tfrac{\Lambda_{i-1}}{\Lambda_j}\big),\lambda_i\Big) 
&\le \Hc\M(\lambda) \sum_{n=1}^{j-1}\frac{\lambda_n}{\Lambda_j}f_j \big(\tfrac{\Lambda_{n-1}}{\Lambda_j}\big)
=\Hc\M(\lambda) \sum_{n=0}^{j-2}\frac{\lambda_{n+1}}{\Lambda_j}f_j \big(\tfrac{\Lambda_n}{\Lambda_j}\big)\\
&=\Hc\M(\lambda) \Big( \frac{\lambda_1}{\Lambda_j}f_j(0)+\sum_{n=1}^{j-2}\frac{\lambda_n}{\Lambda_j}f_j \big(\tfrac{\Lambda_{n}}{\Lambda_j}\big)+\sum_{n=1}^{j-2}\frac{\lambda_{n+1}-\lambda_n}{\Lambda_j}f_j \big(\tfrac{\Lambda_{n}}{\Lambda_j}\big)\Big).
}
Thus by $f_j(0)=y_1$ and $f_j(\tfrac{\Lambda_n}{\Lambda_j})=\chi_{y,\psi}(\tfrac{\Lambda_n}{\Lambda_j})$, we have
\Eq{*}{
\int_{0}^1 C_j(x) dx &\le \frac{\lambda_1}{\Lambda_j} y_1+\Hc\M(\lambda) \Big( \frac{\lambda_1}{\Lambda_j}y_1+\sum_{n=1}^{j-2}\frac{\lambda_n}{\Lambda_j} \chi_{y,\psi} \big(\tfrac{\Lambda_{n}}{\Lambda_j}\big)+\sum_{n=1}^{j-2}\frac{\lambda_{n+1}-\lambda_n}{\Lambda_j}\chi_{y,\psi} \big(\tfrac{\Lambda_{n}}{\Lambda_j}\big)\Big)\\
&\qquad + \sum_{n=1}^{j-1} \frac{\lambda_{n+1}-\lambda_n}{\Lambda_j}\cdot \Mm_{i=1}^{n} \Big(\chi_{y,\psi} \big(\tfrac{\Lambda_{i-1}}{\Lambda_j}\big),\lambda_i\Big).
}
Furthermore monotonicity of $\chi_{y,\psi}$ implies
\Eq{*}{
\sum_{n=1}^{j-2}\frac{\lambda_n}{\Lambda_j}\chi_{y,\psi} \big(\tfrac{\Lambda_{n}}{\Lambda_j}\big) \le
\sum_{n=1}^{j-2} \int_{\frac{\Lambda_{n-1}}{\Lambda_j}}^{\frac{\Lambda_n}{\Lambda_j}} \chi_{y,\psi}(t)\:dt\le
\int_0^1 \chi_{y,\psi}(t)\:dt=\int_0^{\Psi_M} \chi_{y,\psi}(t)\:dt=\sum_{m=1}^M \psi_my_m.
}
Thus, by \eq{Fin1},
\Eq{Aprest}{
\sum_{m=1}^{M} \psi_m \cdot \Mm_{i=1}^m \big(y_i,\psi_i\big) 
&\le P_j+\Hc\M(\lambda)Q_j+R_j+\Hc\M(\lambda) \sum_{m=1}^M \psi_my_m,
}
where
\Eq{*}{
P_j&:=(1+\Hc\M(\lambda)) \frac{\lambda_1}{\Lambda_j} y_1,\\
Q_j&:=\sum_{n=1}^{j-2}\frac{\lambda_{n+1}-\lambda_n}{\Lambda_j}\chi_{y,\psi} \big(\tfrac{\Lambda_{n}}{\Lambda_j}\big),\\
R_j&:=\sum_{n=1}^{j-1} \frac{\lambda_{n+1}-\lambda_n}{\Lambda_j}\cdot \Mm_{i=1}^{n} \Big(\chi_{y,\psi} \big(\tfrac{\Lambda_{i-1}}{\Lambda_j}\big),\lambda_i\Big).
}
We are going to prove that $\lim_{j \to \infty} P_j \le 0$, $\limsup_{j \to \infty} Q_j \le 0$ and $\limsup_{j \to \infty} R_j \le 0$. Having this we take $\limsup$ as $j \to \infty$ in \eq{Aprest} and, by subadditivity of $\limsup$, we obtain desired inequality.

As $\Lambda_j \to \infty$, we have $\lim\limits_{j \to \infty} P_j=0$. In the second term we obtain, by Abel's summation formulae,
\Eq{*}{
Q_j=\frac{\lambda_{j-1}}{\Lambda_j} \chi_{y,\psi}\big( \tfrac{\Lambda_{j-2}}{\Lambda_j}\big)-\frac{\lambda_1}{\Lambda_j} \chi_{y,\psi}\big(\tfrac{\lambda_1}{\Lambda_j}\big) +\sum_{k=2}^{j-2} \frac{\lambda_k}{\Lambda_j} \Big( \chi_{y,\psi}\big(\tfrac{\Lambda_{k-1}}{\Lambda_j}\big)- \chi_{y,\psi}\big(\tfrac{\Lambda_k}{\Lambda_j}\big)\Big)
}
As $y$ is nonincreasing we get $\chi_{y,\psi}\big(\tfrac{\Lambda_{k-1}}{\Lambda_j}\big)-\chi_{y,\psi}\big(\tfrac{\Lambda_k}{\Lambda_j}\big)\ge 0$ and
\Eq{*}{
\sum_{k=2}^{j-2} \frac{\lambda_k}{\Lambda_j} \Big( \chi_{y,\psi}\big(\tfrac{\Lambda_{k-1}}{\Lambda_j}\big)-\chi_{y,\psi}\big(\tfrac{\Lambda_k}{\Lambda_j}\big)\Big) 
&\le\Big( \max_{k\in\{2,\dots,j-2\}} \frac{\lambda_k}{\Lambda_j}\Big) \sum_{k=2}^{j-2}  \Big(\chi_{y,\psi}\big(\tfrac{\Lambda_{k-1}}{\Lambda_j}\big)-\chi_{y,\psi}\big(\tfrac{\Lambda_k}{\Lambda_j}\big)\Big)\\
&= \Big(\max_{k\in\{2,\dots,j-2\}} \frac{\lambda_k}{\Lambda_j}\Big) \Big(\chi_{y,\psi}\big(\tfrac{\lambda_1}{\Lambda_j}\big)-\chi_{y,\psi}\big(\tfrac{\Lambda_{j-1}}{\Lambda_j}\big)\Big)\\
&\le \max_{k\in\{1,\dots,j\}} \frac{\lambda_k}{\Lambda_j} \chi_{y,\psi}\big(\tfrac{\lambda_1}{\Lambda_j}\big).
}
As $y$ is nonincreasing we obtain $\sup \chi_{y,\psi}=y_1$, thus
\Eq{*}{
Q_j \le 2 y_1 \max_{k\in\{1,\dots,j\}} \frac{\lambda_k}{\Lambda_j}.
}
As $\lambda_j/\Lambda_j \to 0$ and $\Lambda_j \to \infty$ we obtain, by Lemma~\ref{lem:maxk1n}, that the limit on the right hand side tends to $0$ therefore
$\limsup\limits_{j \to \infty} Q_j \le 0$. 

To estimate the upper limit of $(R_j)$ observe that, as $f$ in nonincreasing and $\M$ is monotone, by Lemma~\ref{lem:decr} the mapping $n \mapsto \Mm_{i=1}^{n} \big(\chi_{y,\psi}\big(\tfrac{\Lambda_{i-1}}{\Lambda_j}\big),\lambda_i\big)$ is nonincreasing. Therefore we can apply Abel's summation formula again to establish the inequality $\limsup\limits_{j \to \infty} R_j \le 0$. 

Finally, if we consider the upper limit as $j \to \infty$ in \eq{Aprest}, we obtain \eq{E:1gamma}.

\end{document}